\documentclass[11pt]{amsart}
\usepackage{tikz,tikz-cd}
\usepackage{tikz}
\usetikzlibrary{shapes.geometric,positioning}
\usetikzlibrary{arrows,decorations.pathmorphing,decorations.pathreplacing}
\usetikzlibrary{matrix,arrows.meta}
\usepackage{tkz-graph}
\usepackage[all]{xy}
\usepackage{fancyvrb}
\usepackage[normalem]{ulem}
\usepackage{xspace}
\newcommand{\txn}[1]{\textnormal{#1}}
\newcommand{\rsphere}{\mathbb{C}\mathbb{P}^1}
\newcommand{\algbr}{\overline{\mathbb{Q}}}
\newcommand{\absgal}{\txn{Gal}(\algbr/\mathbb{Q})}
\newcommand{\inv}[1]{#1^{-1}}
\newcommand{\gal}[1]{\textnormal{Gal}(#1/\mathbb Q)}
\newcommand{\per}{(\sigma,\alpha,\varphi)}

\newcommand{\belyi}{Bely\u{\i}\xspace}

\newcommand{\zoi}{\{0,1,\infty\}}
\usepackage{color}
\definecolor{blue}{rgb}{0.0, 0.0, 1.0}
\definecolor{candyapplered}{rgb}{1.0, 0.03, 0.0}
\usepackage{graphicx}
\usepackage[all]{xy}
\usepackage{placeins}
\usepackage{enumitem}
\usepackage{amssymb}
\usepackage{latexsym}
\usepackage{amsmath}
\usepackage{mathrsfs}
\usepackage{array,booktabs}
\usepackage{verbatim}
\usepackage{color}
\usepackage[normalem]{ulem}
\usepackage{float}

\usepackage{soul}

\usepackage{color}
\definecolor{teal}{rgb}{0.0, 0.5, 0.5}
\definecolor{tealblue}{rgb}{0.21, 0.46, 0.53}
\definecolor{tealgreen}{rgb}{0.0, 0.51, 0.5}
\definecolor{tuscanred}{rgb}{0.51, 0.21, 0.21}
\definecolor{sangria}{rgb}{0.57, 0.0, 0.04}
\definecolor{rufous}{rgb}{0.66, 0.11, 0.03}
\definecolor{pinegreen}{rgb}{0.0, 0.47, 0.44}
\definecolor{darkscarlet}{rgb}{0.34, 0.01, 0.1}
\definecolor{darkseagreen}{rgb}{0.56, 0.74, 0.56}
\definecolor{darkpastelred}{rgb}{0.76, 0.23, 0.13}
\definecolor{darkpink}{rgb}{0.91, 0.33, 0.5}
\definecolor{darkpastelblue}{rgb}{0.47, 0.62, 0.8}
\definecolor{alizarin}{rgb}{0.82, 0.1, 0.26}
\definecolor{candyapplered}{rgb}{1.0, 0.03, 0.0}

\usepackage{hyperref}

\newcommand{\hyref}[2]{ \hyperref[#2]{#1~\ref*{#2}} }

\theoremstyle{plain}
\newtheorem{theorem}{Theorem}[section]

\newtheorem{lemma}[theorem]{Lemma}

\newtheorem{proposition}[theorem]{Proposition}

\theoremstyle{definition}
\newtheorem{remark}[theorem]{Remark}

\newtheorem{example}[theorem]{Example}
\newtheorem{definition}[theorem]{Definition}

\newtheorem{conv}{Convention}

\DeclareMathAlphabet{\mathpzc}{OT1}{pzc}{m}{it}

\newcommand{\cL}{\mathcal{L}}

\renewcommand{\setminus}{\backslash}

\newcommand{\arr}{\ar@{-}[r]}

\renewcommand{\phi}{\varphi}
\renewcommand{\epsilon}{\varepsilon}

\begin{document}

\date{\today}

\parskip7pt
\parindent0pt

\hyphenation{Gro-then-di-eck Pro-gra-mme Ma-the-mat-i-cal en-fan-ts}

\title{Dessins d'enfants and Brauer configuration algebras}

\thanks{This work has been supported through the EPSRC Early Career Fellowship EP/P016294/1 for the second author and the University of Leicester.}

\author{Goran Mali\'c}
\address{Department of Computer Science, Smith College, Northampton, MA 01063, USA}
\email{goranm00@gmail.com}
\author{Sibylle Schroll} 
\address{Department of Mathematics, University of Leicester, University Road, Leicester LE1 7RH, United Kingdom}
\email{schroll@leicester.ac.uk}

\begin{abstract}
In this paper we associate to a dessin d'enfant an associative algebra, called a Brauer configuration algebra. This is an algebra given by quiver and relations induced by the monodromy of the dessin d'enfant.  We show that the dimension of the  Brauer configuration algebra associated to a dessin d'enfant and the dimension of the centre this algebra are invariant under the action of the absolute Galois group.  We give some examples of well-known algebras and their dessins d'enfants. Finally we show that the Brauer configuration algebras of a dessin d'enfant and its dual share the same path algebra. 
\end{abstract}

\maketitle


\section{Introduction}

In this paper we associate a newly defined class of finite dimensional algebras called Brauer configuration algebras \cite{GreenSchrollBrauerConfig} to a dessin d'enfant. This new algebraic structure generalises Brauer graph algebras, a well-understood class of algebras of tame representation type. Brauer graph algebras arise from clean dessins d'enfants, that is, those dessins d'enfants for which the ramification indices above 1 are all equal to 2; by relaxing this restriction we obtain Brauer configuration algebras. The connection between Brauer graph algebras and clean dessins d'enfants is established in \cite{MS2}. 

There is an action of the absolute Galois group $\rm{Gal}(\overline{\mathbb{Q}}/\mathbb{Q})$ on Brauer configuration algebras induced by the natural action of $\rm{Gal}(\overline{\mathbb{Q}}/\mathbb{Q})$ on dessins d'enfants. We prove that this action keeps invariant certain algebraic data, such as the dimension of the Brauer configuration algebra $A$ associated to a dessin d'enfant, as well as the dimension of the centre of $A$. More invariants in the special case of clean dessins d'enfants and Brauer graph algebras are shown in \cite{MS2}. Our hope is that the results in this paper will open the theory of dessins d'enfants to the techniques of representation theory and homological algebra with the idea of  eventually leading to new Galois invariants.

A dessin d'enfant (dessin for short) is a classical combinatorial object given by a cellular embedding of a bipartite graph on a connected, closed and orientable surface. However, a dessin induces on its underlying surface the structure of an algebraic curve defined over the algebraic numbers $\algbr$ and conversely, any algebraic curve over $\algbr$ corresponds to a dessin. Consequently, there is a natural faithful action of the absolute Galois group $\absgal$ on the set of dessins; both results are due to \belyi \cite{belyi80}. Dessins were introduced by Grothendieck in his influential Esquisse d'un Programme \cite{Grothendieck} where they are considered as a starting point to the study of the structure of $\absgal$ as the automorphism group of a tower of so-called Teichm\"uller groupoids of moduli spaces of curves in any genus. One of the main goals of the theory of dessins d'enfants is to understand the invariants of this action and consequently to be able to distinguish between any two orbits. Several invariants that distinguish between certain orbits are known, but a complete answer does not seem to be within reach.

Brauer configuration algebras are a new class of associative symmetric algebras (equipped with a non-degenerate symmetric form) generalising  Brauer graph algebras \cite{GreenSchrollBrauerConfig}. Brauer graph algebras originate in the modular representation theory of finite groups \cite{Janusz} and their representation theory is well understood. While Brauer graph algebras are of tame representation type, that is, the isomorphism classes of their indecomposable representations can be parametrised by finitely many one parameter families, Brauer configuration algebras are, in general, of wild representation type. An algebra $A$ is of wild representation type if for any finite dimensional algebra there exists a representation embedding into the module category of $A$, that is in some sense their representation theory contains the representation theory of any finite dimensional algebra.  It is generally believed that the finite dimensional representations of a wild algebra cannot be classified in a meaningful way. However, for the class of Brauer configuration algebras,  a precise structural description of all representations has been given in \cite{GreenSchrollAdvances}. Furthermore, the dimensions of both the Brauer configuration algebras \cite{GreenSchrollBrauerConfig} and of the zeroth Hochschild cohomology space of a Brauer configuration algebra \cite{Sierra} are known.

We start with a rapid introduction to both Grothendieck's theory of dessins d'enfants in Section \ref{dessins} and to associative algebras given by quivers and relations in Section 3.  In Section 4 we introduce Brauer configuration algebras defined in \cite{GreenSchrollBrauerConfig} and show how they naturally arise from  dessins d'enfants. In Section 5 we give two examples of well-known algebras that arise as Brauer configuration algebras of specific dessins d'enfants, the symmetric Nakayama algebras and quotients of preprojective algebras of type $\tilde{A}$.  In Section 6 we discuss the Brauer configuration algebra associated to  the dessin $D^*$ dual to $D$ and show that $D^*$ and $D$ have the same quiver when $D$ has no faces and black vertices of degree 1. 
In  Section 7 we discuss the Galois action on  Brauer configuration algebras, and show that the  centre of a Brauer configuration algebra as well as the  dimension of the  Brauer configuration algebra are invariant under the action of $\absgal$.

\textbf{ Acknowledgments: } The first author would like to thank the University of Leicester for hospitality during his stays at Leicester. 

\section{Dessins d'enfants}\label{dessins}

We start by briefly recalling Grothendieck's theory of dessins d'enfants. For a more detailed introduction to the topic see for example \cite{GGD, JonesWolfart, LandoZvonkin, schneps94}.  We closely follow the exposition in \cite{Malic1}.

A \emph{dessin d'enfant} (dessin for short) is a finite bipartite connected graph $G$ (with multiple edges allowed) embedded cellularly on a connected, closed and orientable surface $X$.  In this embedding the vertices of $G$ are points on $X$ coloured in black or white, the edges are curved segments on $X$ intersecting only at the vertices so that each edge ends in exactly one black and one white vertex, and the complement of the embedding is a finite union of connected components homeomorphic to an open 2-cell, which we call \emph{faces} of the dessin. Two dessins $(G_1,X_1)$ and $(G_2,X_2)$ are isomorphic if there is an orientation preserving homeomorphism $X_1\to X_2$ that restricts to a bipartite graph isomorphism $G_1\to G_2$.
 
Equivalently, a dessin is a pair $(X,f)$ where $X$ is a compact Riemann surface and $f\colon X \to\rsphere$ is a holomorphic ramified covering of the Riemann sphere, ramified over a subset of $\zoi$. The pair $(X,f)$ is also called a \emph{\belyi pair} and the map $f$ is called a \emph{\belyi map} or a \emph{\belyi function}. Two dessins $(X_1,f_1)$ and $(X_2,f_2)$ are isomorphic if they are isomorphic as ramified covers, i.e.\ if there is an orientation preserving homeomorphism $h\colon X_1 \to X_2$ such that $f_1=f_2\circ h$.

The equivalence between the two definitions is obtained as follows: given a dessin $(X,f)$, the preimage $f^{-1}([0,1])$ of the closed unit interval corresponds to a cellular embedding of a bipartite graph on the underlying topological surface of $X$ such that the black and white vertices correspond to $\inv f(0)$ and $\inv f(1)$ respectively, and the edges correspond to the preimages of the open unit interval.

Conversely, given a dessin on a topological surface $X$, add a single new vertex to  the interior of each face.  To distinguish it from the black and white vertices, we will represent these vertices by  diamonds $\diamond$. Now triangulate $X$ by connecting the diamonds with the black and white vertices that are on the boundaries of the corresponding faces. Following the orientation of $X$, call the triangles with vertices oriented as $\bullet$-$\circ$-$\diamond$-$\bullet$ positive, and call other triangles negative (see Figure \ref{triangulation}).
\begin{figure}[ht]
  \centering\includegraphics[scale=.5, trim={8cm 9.5cm 3.5cm 5cm}]{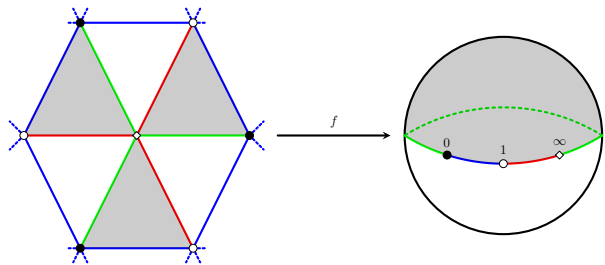}
  \caption{The positive (shaded) and negative triangles are mapped to the upper and lower-half plane, respectively. The sides of the triangles are mapped to $\mathbb R\cup\{\infty\}$ so that the black and white vertices map to 0 and 1, respectively, and the face centres map to $\infty$.}\label{triangulation}
\end{figure}

Now send the positive and negative triangles to the upper and lower half-plane of $\mathbb C$, respectively, and send the sides of the triangles to the real line so that black, white and diamond vertices are sent to $0$, $1$ and $\infty$, respectively. As a result, we obtain a ramified cover $f\colon X\to\rsphere$, ramified only over a subset of $\{0,1,\infty\}$. We now impose on $X$ the unique Riemann surface structure which makes $f$ holomorphic.

\subsection{A permutation representation of dessins}\label{subsection:perm_rep}

The monodromy action induced by lifts under $f$ of simple closed loops on $\rsphere$ based at $1/2$ and circling around 0 and 1 gives one further definition of dessins via group theory: a dessin is the conjugacy class of a 2-generated transitive subgroup $\langle\sigma,\alpha\rangle$ of $S_n$, where $n$ is the degree of $f$. Two dessins $\langle\sigma_1,\alpha_1\rangle$ and $\langle\sigma_2,\alpha_2\rangle$ are isomorphic if they are isomorphic as permutation groups. 

Throughout this section let $(X,f)$ be a dessin with $n$ half-edges (or, equivalently, such that $f$ is a degree $n$ covering). In this section we describe how each such dessin can be represented by a triple $\per$ of permutations in the symmetric group $S_{n}$. However, we first fix the following notation which we will keep throughout the paper.

\begin{conv}\label{conv::edge labels}We label the half-edges of a dessin with the elements of the set $\{1,\dots,n\}$ so that, when standing at a black vertex, and looking towards an adjacent white vertex, the label is placed on the `left side' of the half-edge. See Figure \ref{planargraph1} for an example.
\end{conv}
\begin{figure}[ht]
  \centering
	\includegraphics[scale=.7, trim={7cm 20.5cm 7cm 5cm}]{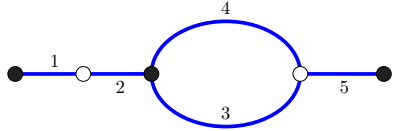}
  \caption{Labelling of half-edges. The labels are always on the left when looking from a black vertex to its adjacent white vertices.}\label{planargraph1}
\end{figure}

Following convention~\ref{conv::edge labels}, label the half-edges of a dessin arbitrarily. Now let $\sigma$ and $\alpha$ denote the permutations which record the cyclic (counter-clockwise) orderings of the labels around black and white vertices, respectively, and let $\varphi$ denote the permutation which records the counter-clockwise ordering of the labels within each face.
\begin{example}For the dessin in Figure \ref{planargraph1} we have $\sigma=(1)(2\,3\,4)(5)$, $\alpha=(1\,2)(3\,5\,4)$ and $\varphi=(1\,4\,5\,2)(3)$. The cycles of length 1 are usually dropped. Note that the cycle corresponding to the `outer face' is, from the reader's perspective, recorded clockwise. This does not violate our convention since the labels of that face should be viewed from the opposite side of the sphere and switch orientation once the face is unfolded into a disc.
\end{example}

A change of  labels corresponds to simultaneous conjugation of $\sigma$, $\alpha$ and $\varphi$ by some element in $S_n$. Therefore, any dessin can be represented, up to conjugation, by a triple of permutations.

\begin{definition}The length of a cycle in $\sigma$ or $\alpha$ corresponding to a black or a white vertex, respectively, is called the \emph{degree} of the vertex. The length of a cycle in $\varphi$ corresponding to a face is called the \emph{degree} of the face. Thus, the degree of a vertex is the number of half-edges incident to it, while the degree of a face is half the number of half-edges on its boundary.
\end{definition}

A triple $(\sigma,\alpha,\varphi)$ representing a dessin $D=(X,f)$ satisfies the following properties:
\begin{itemize}
  \item the group $\left\langle\sigma,\alpha,\varphi\right\rangle$ acts transitively on the set $\{1,\dots,n\}$ and
  \item $\sigma\alpha\varphi=1$.
\end{itemize}
The first property above is due to the fact that dessins are connected while the second is due to the following: consider three non-trivial simple loops $\gamma_0$, $\gamma_1$ and $\gamma_\infty$ in $\pi_1(\rsphere\setminus\{0, 1, \infty\},1/2)$ based at $1/2$ and going around 0, 1 and $\infty$ once, respectively. The lifts of these loops under $f$ correspond to paths on $X$ that start and end at a (possibly the same) point in $\inv f(1/2)$ . We observe the following.
\begin{itemize}
  \item Every half-edge of $D$ contains precisely one element of $\inv f(1/2)$ since $f$ is unramified at $1/2$.
  \item The cardinality of $\inv f(1/2)$ is precisely $n$. Hence there is a bijection $\inv f(1/2)\to\{1,\dots,n\}$.
  \item With respect to this bijection, $\sigma$, $\alpha$ and $\varphi$  can be thought of as permutations of $\inv f(1/2)$.
\end{itemize}
Therefore the loops $\gamma_0$, $\gamma_1$ and $\gamma_\infty$ induce $\sigma$, $\alpha$ and $\varphi$. Since the product $\gamma_0\gamma_1\gamma_\infty$ is trivial in $\pi_1(\rsphere\setminus\zoi,1/2)$, the corresponding permutation $\sigma\alpha\varphi$ must be trivial as well.

We have now seen that to every dessin with $n$ half-edges we can assign a triple of permutations in $S_n$ such that their product is trivial and the group that they generate acts transitively on the set $\{1,\dots, n\}$. 

In a similar way we can show that this assignment works in the opposite direction as well: given three permutations $\sigma$, $\alpha$ and $\varphi$ in $S_n$ such that $\sigma\alpha\varphi=1$ and such that the group that they generate acts transitively on $\{1,\dots,n\}$, we can construct a dessin with $n$ half-edges respecting the order imposed by the permutations $\sigma$, $\alpha$ and $\varphi$. Therefore, up to simultaneous conjugation, a dessin is uniquely represented by a transitive triple $(\sigma,\alpha,\varphi)$ with $\sigma\alpha\varphi=1$, and such a triple recovers a unique dessin up to isomorphism.

\begin{remark} Dessins correspond to $2$-generated transitive permutation groups since we can set $\varphi=\inv{(\sigma\alpha)}$. However, we prefer to emphasise all three permutations.
\end{remark}
We will use the notation $D=\per$ to denote that a dessin $D$ is represented by the triple $\per$.

\subsection{\belyi's theorem}\label{subsection:Belyi's theorem}
\belyi's theorem is the starting point of Grothendieck's remarkable Esquisse d'un Programme \cite{Grothendieck} in which he sketches an approach towards understanding  the absolute Galois group $\absgal$ over the rationals as an automorphism group of a certain topological object. We restate the theorem here.

\begin{theorem}[\belyi]\label{belyi}Let $X$ be a smooth projective algebraic curve defined over $\mathbb C$. Then $X$ is defined over $\algbr$ if and only if there is a holomorphic ramified covering $f\colon X\to\rsphere$, ramified at most over a subset of $\{0,1,\infty\}$.
\end{theorem}

Aside from \belyi's own papers \cite{belyi80, belyi02}, various other proofs can be found in, for example, \cite[theorem 4.7.6]{szamuely09} or \cite[chapter 3]{GGD} or the recent proof in \cite{goldring14}. \belyi himself concluded that the above theorem implies that $\absgal$ embeds into the outer automorphism group of the profinite completion of the fundamental group of $\rsphere\setminus\{0,1,\infty\}$. However it was Grothendieck who observed that $\absgal$ must therefore act faithfully on the set of dessins as well. This interplay between algebraic, combinatorial and topological objects is what prompted Grothendieck to develop Esquisse d'un Programme. For more detail, see \cite{schneps_lochak97vol1} or \cite{szamuely09}.

\subsection{Galois action on dessins}\label{subsection: Galois action}

Let $D=(X,f)$ be a dessin. If $X$ is of genus 0, then necessarily $X=\rsphere$ and $f\colon\rsphere\to\rsphere$ is a rational map with critical values in the set $\{0,1,\infty\}$. If $f=p/q$, where $p,q \in\mathbb C[z]$, then \belyi's theorem implies that $p,q\in\algbr[z]$. Moreover, the coefficients of both $p$ and $q$ generate a finite Galois extension $K$ of $\mathbb Q$. Therefore $p,q\in K[z]$. Then $\gal K$ acts on $f$ by acting on the coefficients of $p$ and $q$, that is if $\theta\in\gal K$ and
\begin{align*}
  f(z) &= \frac{a_0+a_1z+\cdots+a_mz^m}{b_0+b_1z+\cdots b_nz^n},\\[.3cm]
 \txn{then }f^\theta(z) &= \frac{\theta(a_0)+\theta(a_1)z+\cdots+\theta(a_m)z^m}{\theta(b_0)+\theta(b_1)z+\cdots \theta(b_n)z^n}.
\end{align*}

If $X$ is of positive genus, then as an algebraic curve it is defined by the zero-set of an irreducible polynomial $F$ in $\mathbb C[x,y]$. This time we must take into consideration the coefficients of both $F$ and $f$ which, due to \belyi's theorem again, generate a finite Galois extension $K$ of $\mathbb Q$. Similarly as in the genus 0 case, $\gal K$ acts on $D$ by acting on the coefficients of both $F$ and $f$ simultaneously.

It is not immediately clear that the action of some automorphism in $\gal K$ on a \belyi map $f$ will produce a \belyi map. However, this is indeed  the case and we refer the reader to the discussion in \cite[ch.\ 2.4.2]{LandoZvonkin}.

Since any $\mathbb Q$-automorphism of $K$ extends to an $\mathbb Q$-automorphism of $\algbr$ \cite[ch.\ 3]{bor_jan}, we truly have an action of $\absgal$ on the set of dessins.

\begin{definition}We will denote by $D^\theta=(X^\theta,f^\theta)$ the dessin that is the result of the action of $\theta\in\absgal$ on $D=(X,f)$. We will also say that $D^\theta$ is \emph{conjugate} to $D$.\end{definition}

The following example is borrowed from \cite[ex.\ 2.3.3]{LandoZvonkin}.

\begin{example}\label{exmp1}Let $D=(X,f)$ be a dessin where $X$ is the elliptic curve
\[y^2=x(x-1)(x-(3+2\sqrt 3)),\]
and $f\colon X\to \rsphere$ is the composition $g\circ \pi_x$, where $\pi_x \colon X\to\rsphere$ is the projection to the first coordinate and $g\colon\rsphere\to\rsphere$ is given by
\[g(z)=-\frac{(z-1)^3(z-9)}{64z}.\]
The corresponding bipartite map is depicted on the left in Figure \ref{genus1action}.

\begin{figure}[ht]
  \centering
    \includegraphics[scale=.65,trim=7cm 14cm 7cm 5cm,scale=.9]{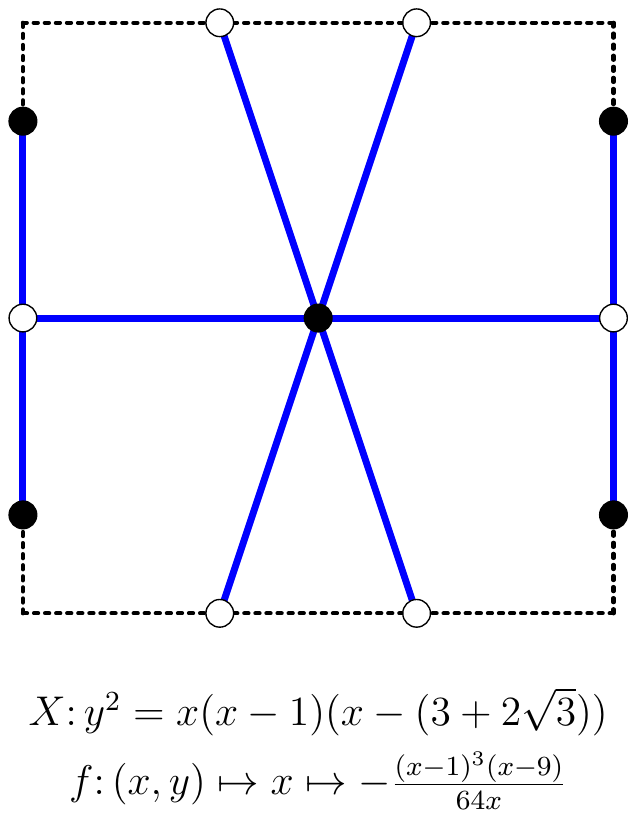}
    \includegraphics[scale=.65,trim=7cm 14cm 7cm 5cm,scale=.9]{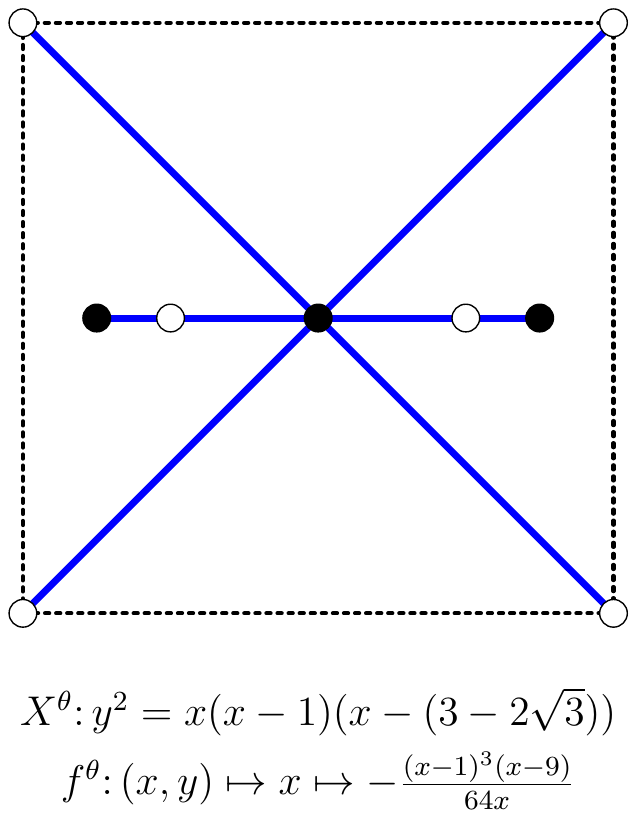}
  \caption{The two dessins $(X,f)$ and $(X^\theta,f^\theta)$ from example \ref{exmp1}. The dotted lines indicate the boundary of the polygon representation of a genus 1 surface with the usual identification of the left-and-right and top-and-bottom sides.}\label{genus1action}
\end{figure}

Note that we must consider $g\circ \pi_x$ and not just $\pi_x$ since $\pi_x$ is not a \belyi map; it is ramified over four points, namely $0$, $1$, $3+2\sqrt 3$ and $\infty$. However, $g$ maps these four points onto the set $\{0,1,\infty\}$ and therefore $g\circ\pi_x$ is a true \belyi map.

The Galois extension that the coefficients of $X$ and $f$ generate is $K=\mathbb Q(\sqrt 3)$ and the corresponding Galois group has only one non-trivial automorphism, namely $\theta\colon\sqrt3\mapsto-\sqrt3$. Therefore $X^\theta$ is the elliptic curve $y^2=x(x-1)(x-(3-2\sqrt 3))$. The curve $X^\theta$ is non-isomorphic to $X$, which can easily be seen by computing the $j$-invariants of both.

What about $f^\theta$? In this case, $\pi_x\colon X^\theta\to\rsphere$ is unramified over $3+2\sqrt 3$ and ramified over $3-2\sqrt 3$. However, $g$ maps $3-2\sqrt 3$ to 0 as well, and since $g$ is defined over $\mathbb Q$, the \belyi functions $f$ and $f^\theta$ coincide.

The bipartite map corresponding to $(X^\theta,f^\theta)$ is depicted on the right in Figure \ref{genus1action}.
\end{example}

This action of $\absgal$ on dessins is faithful already on the set of \emph{trees}, that is the genus 0 dessins with precisely one face and Shabat polynomials as \belyi functions. However, this is not straight-forward (proofs can be found in \cite{schneps94, GGD}) and, surprisingly, it is much easier to show faithfulness in genus 1 \cite[ch. 4.5.2]{GGD}. Moreover, the action is faithful in every genus \cite[ch. 4.5.2]{GGD}.

\section{From a dessin d'enfant to a path algebra}\label{dessin to path}

Recall that a quiver $Q = (Q_0, Q_1,s,t)$ is given by a finite set of vertices $Q_0$ and a finite set of oriented edges $Q_1$, called arrows, and functions $s,t : Q_1 \to Q_0$ where for $a \in Q_1$, $s(a)$ denotes the start of $a$ and $t(a)$ denotes its target.

Consider a dessin $D$ in its permutation representation $\per$.  The permutation $\sigma$ induces a quiver $\mathring Q$ in the following way. The vertices $\mathring Q_0$ of the quiver correspond to the white vertices of $D$ and the set of arrows $\mathring Q_1$ is induced by the permutation $\sigma$ in the following way: if $(i_1 \cdots i_k)$ is a non-trivial cycle in $\sigma$ and if $\alpha_{i_1}, \ldots, \alpha_{i_k}$ denote the vertices of the quiver corresponding to the white vertices in $D$ incident with $i_1, \ldots, i_k$ respectively,  then we define a cycle of arrows in the quiver by setting  $\alpha_{i_1} \to \alpha_{i_2}  \to \cdots \to \alpha_{i_k} \to \alpha_{i_1}$. Note that the trivial cycles of $\sigma$ induce loop arrows in $\mathring Q$.

\begin{example}The dessin in Figure \ref{fig:example1} is given by the permutations
\begin{align*}
\sigma &= (1~2~3~4~5)(6~7)(8~9)(10)(11),\\
\alpha &= (1)(2~10~11)(3~6~9)(4~5)(7~8)
\end{align*}
in $S_{11}$. The corresponding quiver is shown on the right of Figure \ref{fig:example1}. The differently coloured arrows correspond to the different cycles of the permutation $\sigma$.

\begin{figure}[ht]
	\centering
	\includegraphics[scale=.6]{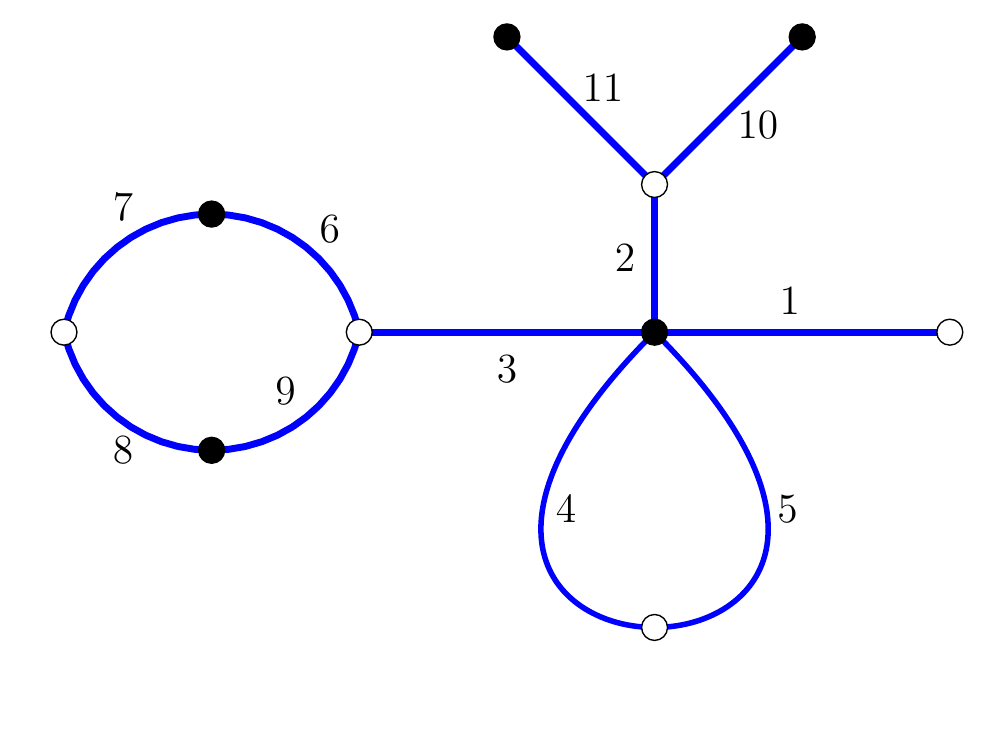}
	\includegraphics[scale=.6]{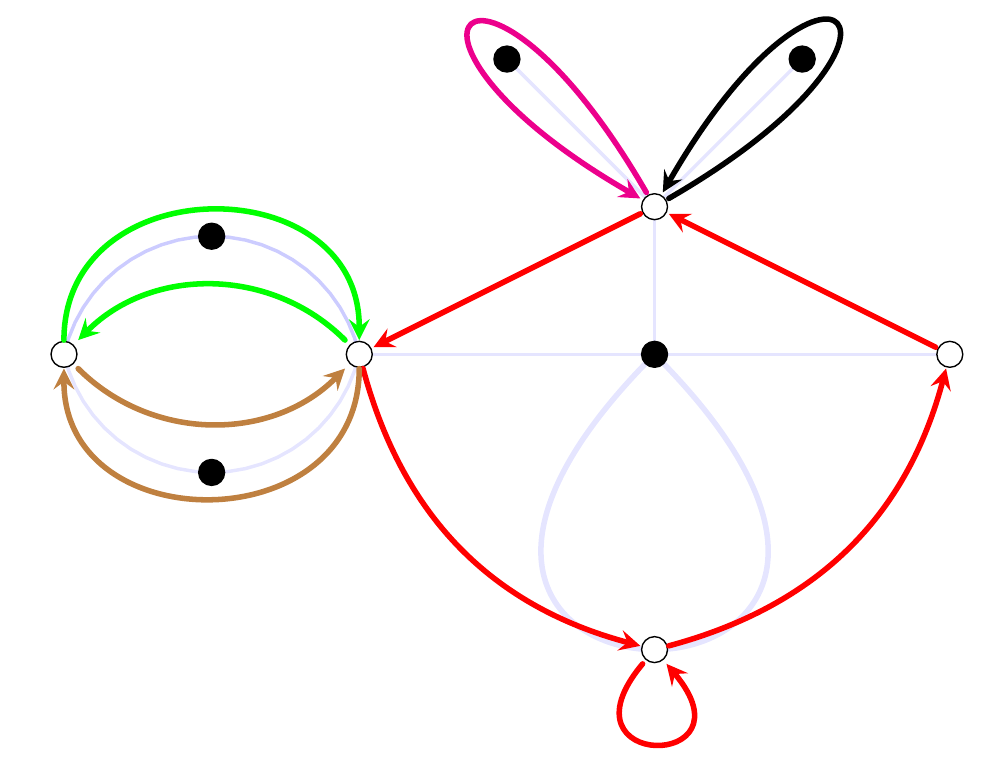}
	\caption{A dessin $D$ (left) and its quiver on the right. The differently coloured arrows in the quiver correspond to the different non-trivial cycles of the permutation $\sigma$ of $D$.}
\label{fig:example1}
\end{figure}
\end{example}

Let $K$ be a field. Recall that given a  quiver $Q$, the {\it path algebra} $KQ$, has vector space basis given by all possible finite paths in $Q$ including a trivial path $e_j$ for every vertex $j \in Q_0$. The  multiplication of two paths is given by concatenation if possible and zero otherwise.  
The multiplicative identity of $KQ$ is given by $\sum_{j \in Q_0} e_j$, the sum  of the trivial paths in $Q$.
If $Q$ has an oriented cycle then $KQ$ is an infinite dimensional algebra. For further details on paths algebras of quivers, see for example  \cite{Assem, Schiffler}.

\begin{example}The quiver $\mathring Q$ associated to the trivial dessin $D=(\rsphere,z)$ is given by a single vertex and a single loop arrow $a$. A basis for $K\mathring Q$ is given by the paths $\{e, a, a^2, \dots, a^n, \dots\}$  and $K\mathring Q$ is therefore isomorphic to the polynomial algebra $K[x]$.
\end{example}

Quivers associated to dessins always have at least one oriented cycle and therefore the associated path algebras are always infinite dimensional.

\subsection{Admissible ideals of $KQ$ and bound quiver algebras}

Let $Q$ be a quiver and $\mathcal R$ be the two-sided ideal of $KQ$ generated by all arrows in $Q$. This ideal decomposes as
\[KQ_1\oplus\cdots\oplus KQ_l\oplus\cdots,\]
where $KQ_l$ is the subspace of $KQ$ generated by all paths of length $l$. The $l$\emph{-th power} $\mathcal R^l$ \emph{of} $\mathcal R$ is the subspace with a basis of all paths of length at least $l$ with decomposition
\[\mathcal R^l = \bigoplus_{r\geq l}KQ_r.\]

\begin{definition}We say that a two-sided ideal $I$ of $KQ$ is \emph{admissible} if there exists $n\geq 2$ such that
\[\mathcal R^n \subseteq I \subseteq \mathcal R^2.\]
If $I$ is an admissible ideal of $KQ$ we say that the quotient algebra $KQ/I$ is a \emph{bound quiver algebra}.
\end{definition}

Bound quiver algebras are finite dimensional. They are indecomposable if and only if the quiver is connected.

The introduction of admissible ideals and bound quiver algebras is far from arbitrary. It is a well known result in representation theory due to Gabriel \cite{Gabriel} stating that if $K$ is algebraically closed, every connected finite dimensional $K$-algebra is Morita equivalent to a bound algebra $KQ/I$ for a unique quiver $Q$ and an admissible ideal $I$ of $KQ$.

\section{From a dessin d'enfant to a Brauer configuration algebra}\label{dessin to brauer}

In this section, given a dessin $D$ we will define a quiver $Q_D$ and an admissible ideal $I_D$ of $K Q_D$ so that the quotient algebra $KQ_D/I_D$ is finite dimensional. 
Let $D$ be a dessin with $m$ black vertices $\sigma_j$, for $j=1,\dots,m$ and $n$ white vertices $\alpha_k$,  for $k=1,\dots,n$. Let $\cL$ be the set of loop arrows in $\mathring Q_1$ induced by a black vertex of degree 1. We will refer to the elements in $\cL$ as \emph{formal loop arrows}. We recall  that the quiver $\mathring Q$ has vertices corresponding to the white vertices of $D$ and if $D= \per$ then the arrows of $\mathring Q$ are induced by the permutation $\sigma$. 

\begin{definition}The quiver $Q_D=(Q_0,Q_1)$ associated to $D$ is the quiver $Q_D=(Q_0,Q_1)$ with $Q_0=\mathring Q_0$ as its vertex set and $Q_1=\mathring Q_1\setminus \cL$ as its arrow set.\end{definition}

\begin{definition} Let $\sigma_j$ be a black vertex of degree $ \deg\sigma_j\geq2$. We call a cycle $a_{i_1}a_{i_2}\cdots a_{i_{\deg\sigma_j}}$ in $Q_D$ induced by $\sigma_j$  a \emph{special} $\sigma_j$\emph{-cycle}.  When a black vertex is unspecified, the corresponding cycle in $Q_D$ will be called a special $\sigma$-cycle. Black vertices of degree 1 by construction of $Q_D$ do not contribute any cycles.

\begin{figure}[ht]
\centering
\includegraphics[scale=.5]{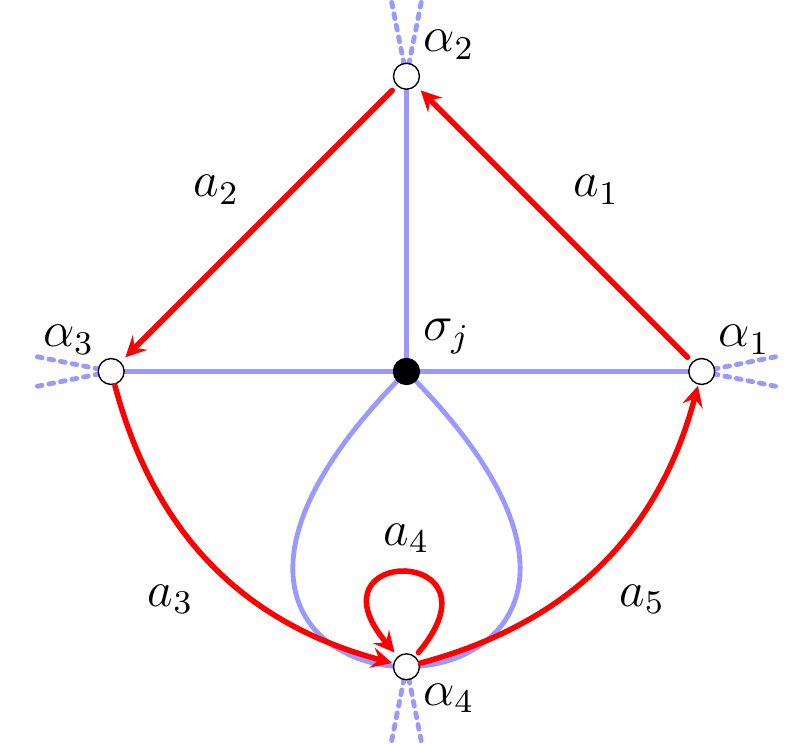}
\caption{A vertex $\sigma_j$ of degree 5 in a dessin with the corresponding cycle in $Q_D$. Note that  $a_4$ is a loop since $\alpha_4$ shares two (consecutive) edges with $\sigma_j$.}%
\label{fig:example2}%
\end{figure}

Furthermore, a special $\sigma$-cycle starting at the white vertex $\alpha_k$ is called \emph{a special} $\sigma$\emph{-cycle at} $\alpha_k$.
\end{definition}

\begin{example} Suppose that $\sigma_j$ is a black vertex of degree $5$ such that the white vertices incident to it are labelled by $\alpha_1, \dots, \alpha_{4}$ in the counter-clockwise order, as in Figure \ref{fig:example2}. Furthermore, let $a_1, \ldots, a_5$ be the arrows as in Figure~\ref{fig:example2}. Then the special $\sigma_j$-cycles in Figure \ref{fig:example2} are given by $a_1a_2a_3a_4a_5$, $a_2a_3a_4a_5a_1$, $a_3a_4a_5a_1a_2$, $a_4a_5a_1a_2a_3$ and $a_5a_1a_2a_3a_4$. For the vertices  $\alpha_1$, $\alpha_3$ and $\alpha_4$ there is exactly one special $\sigma_j$-cycle at $\alpha_1$, $\alpha_2$ and $\alpha_3$, respectively given by
\begin{align*}
\mbox{at } \alpha_1: a_1a_2a_3a_4a_5,\\
\mbox{at } \alpha_2: a_2a_3a_4a_5a_1,\\
\mbox{at } \alpha_3: a_3a_4a_5a_1a_2.
\end{align*}
However, there are two special $\sigma_j$-cycles at $\alpha_4$ given by $a_4a_5a_1a_2a_3$ and $a_5a_1a_2a_3a_4$.
\end{example}

The special $\sigma$-cycles at $\alpha_s$ will contribute relations to \emph{the generating set of relations} $\rho_D$ of \emph{the admissible ideal} $I_D$ of the path algebra $KQ_D$. The set $\rho_D$ consists of three types of relations:

\subsubsection*{Relations of type one.} For each white vertex $\alpha_s$ and each pair $\sigma_j$ and $\sigma_k$ of black vertices of degree at least 2 incident to $\alpha_s$, all relations of the form
\[C_j - C_k,\]
where $C_j$ and $C_k$ are the special $\sigma_j$ and $\sigma_k$-cycles at $\alpha_s$ are in $\rho_D$.

\subsubsection*{Relations of type two.} For all $\sigma_r$ all relations of the type $Ca$ are in $\rho_D$, where $C$ ranges across all special $\sigma_r$-cycles and $a$ is the first arrow of $C$.

\subsubsection*{Relations of type three.} All paths $ab$ of length 2 which are not subpaths of any special cycle are relations in $\rho_D$.

\begin{example}\label{example3} Let $D$ be as the dessin from Figure \ref{fig:example1} with an additional half-edge attached to white vertex $\alpha_1$. Label the arrows as in Figure \ref{fig:example3}.

\begin{figure}[ht]
\centering
\includegraphics[scale=.6]{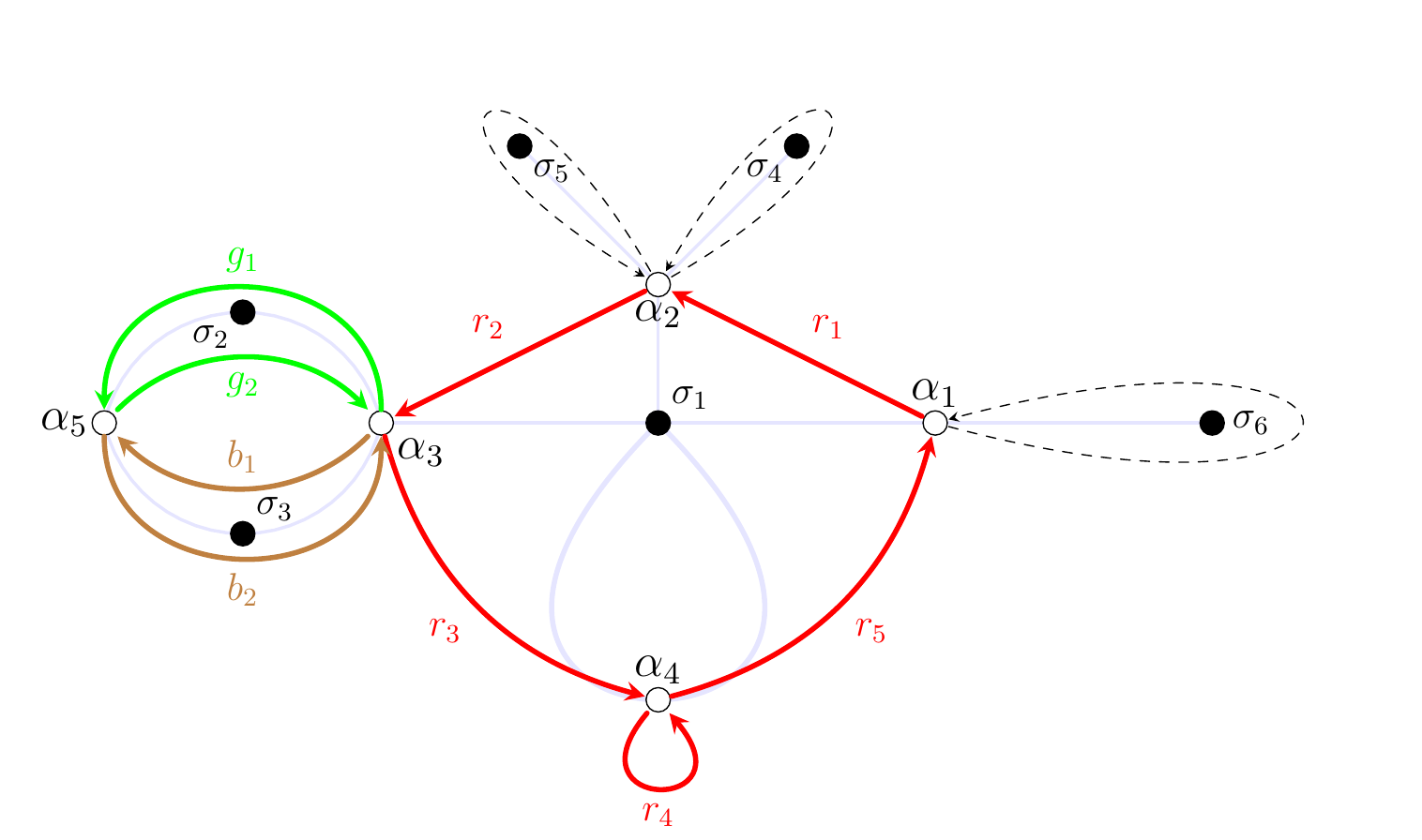}
\caption{The quiver $Q_D$ from Example \ref{example3}. Arrows are labelled according to colour ($r$ for red, $g$ for green, $b$ for brown). The dashed loop arrows are formal.}
\label{fig:example3}%
\end{figure}

	Relations of type one are given by:
	\begin{align*}
		\alpha_1\colon & \textnormal{none.}\\
		\alpha_2\colon & \textnormal{none.}\\
		\alpha_3\colon & r_3r_4r_5r_1r_2 - g_1g_2, r_3r_4r_5r_1r_2 - b_1b_2, g_1g	_	2-	b_1b_2.\\
	\alpha_4\colon & r_4r_5r_1r_2r_3-r_5r_1r_2r_3r_4.\\
	\alpha_5\colon & b_2b_1 - g_2g_1.
	\end{align*}

	Relations of type two are given by:
	\begin{align*}
			\sigma_1\colon & \overbrace{r_1r_2r_3r_4r_5r_1}^{\sigma_1\txn{-cycle at }\alpha_1}, \overbrace{r_2r_3r_4r_5r_1r_2}^{\sigma_1\txn{-cycle at }\alpha_2}, \overbrace{r_4r_5r_1r_2r_3r_4}^{\sigma_1\txn{-cycle at }\alpha_3}, \overbrace{r_4r_5r_1r_2r_3r_4, r_5r_1r_2r_3r_4r_5}^{\sigma_1\txn{-cycles at }\alpha_4}\\
	\sigma_2\colon & \underbrace{g_1g_2g_1}_{\sigma_2\txn{-cycle at }\alpha_3}, \underbrace{g_2g_1g_2.}_{\sigma_2\txn{-cycle at }\alpha_5}\\
	\sigma_3\colon & \underbrace{b_1b_2b_1}_{\sigma_3\txn{-cycle at }\alpha_3}, \underbrace{b_2b_1b_2.}_{\sigma_3\txn{-cycle at }\alpha_5}\\
	\sigma_4, \sigma_5, \sigma_6\colon & \textnormal{none}.
	\end{align*}

	Relations of type three are given by:
	\begin{align*}
	r_2g_1, r_2b_1, r_3 r_5, r_4^2, g_1b_2, g_2b_1, g_2r_3, b_1g_2, b_2g_1, b_2r_3.
	\end{align*}
\end{example}

The bound quiver algebra $KQ_D/I_D$ is called a \emph{Brauer configuration algebra}.  More generally, a \emph{Brauer configuration} as defined in \sloppy \cite{GreenSchrollBrauerConfig} is a tuple $(\Gamma_0,\Gamma_1,\mu,\mathfrak o)$ where $\Gamma_0$ is the set of vertices, $\Gamma_1$ the set \emph{polygons}, i.e.\ multisets of vertices, $\mu\colon\Gamma_0\to\mathbb N$ a function into the positive integers called the \emph{multiplicity}, and $\mathfrak o$ is the function specifying at every vertex a cyclic ordering of the  polygons incident with that vertex. A Brauer configuration algebra is constructed from a Brauer configuration via a quiver and an ideal generated by relations of type one, two and three as defined above, with additional constraints on the relations given by the multiplicity function $\mu$. 

In the language of dessins d'enfants, $\Gamma_0$ is the set of black vertices, $\Gamma_1$ is the set of white vertices, $\mu\colon\Gamma_0\to\mathbb N$ is the function $\mu(\sigma_j)=1$, and $\mathfrak o$ is the counter-clockwise orientation induced by the orientation of the underlying Riemann surface of $D$. When a dessin is \emph{clean}, i.e.\ when all white vertices have degree 2, then the corresponding Brauer configuration algebra is a Brauer graph algebra; the relation between clean dessins and Brauer graph algebras has been studied in \cite{MS2}.

The relations $\rho_D$ are not necessarily a minimal set of relations for $I_D$. The relations of type one and three are always minimal, however relations of type two are often redundant; this is a  generalisation of a similar result for Brauer graph algebras \cite{GSS}. 

\begin{remark}\label{remark:truncated}The reason why the loop arrows induced by black vertices of degree 1 are not included in $Q_D$ is because the ideal $I_D$ would no longer be admissible if the corresponding relations were to be added to $\rho_D$. However, the quotient algebra $K\mathring Q_D/I$, where $I$ is the ideal generated by the relations of type 1, 2 and 3,  including the relations induced by formal loop arrows where each loop is its own special cycle, is isomorphic to the Brauer configuration algebra $KQ_D/I_D$. 

An example of this are the two different  Brauer configuration algebras associated to the dessin in Figure~\ref{fig:example1} to which we can associate a quiver $Q_1$ with formal loop arrows as in Figure~\ref{fig:example1} and a quiver $Q_2$ without the formal loop arrows as in Figure~\ref{fig:example2}.  The algebras  $K\mathring Q_{D_1}/I_1$ and $K\mathring Q_{D_2}/I_2$, where $I_1$ and $I_2$ are the ideals generated by the relations of type 1, 2 and 3 where in the case of $I_1$ every formal loop arrow is its own special cycle. 

Note that black vertices of degree 1 correspond to \emph{truncated vertices} in \sloppy \cite{GreenSchrollBrauerConfig}.
 \end{remark}

\section{Examples of (well-known) algebras and their dessins}

\subsection{Symmetric Nakayama algebras}

Let  $D_n=(\rsphere,f_n)$ be the family of dessins given by the \belyi maps $f_n\colon z\mapsto z^n$ for $n\geq2$.
	The quiver $Q_{D_n}$ is represented by an oriented regular $n$-gon as in Figure \ref{fig:nakayama}, with relations $a_1\cdots a_n a_1$, $\dots$, $a_na_1\cdots a_n$. In this case the only relations in $I_{D_n}$ are relations of type two and they are necessarily minimal.
	\begin{figure}[ht]
	\centering
	\includegraphics[scale=.6]{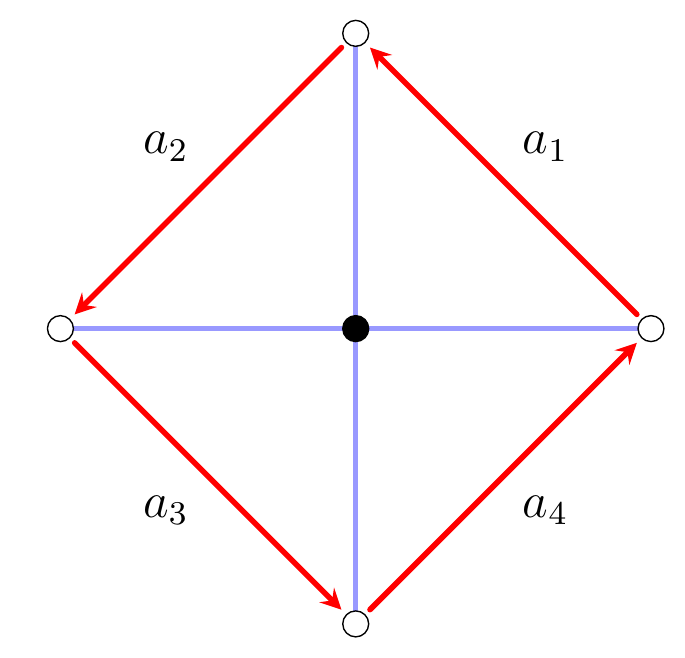}
	\caption{The quiver $Q_{D_4}$ corresponding to the dessin $D_4=(\rsphere,z^4)$.}%
	\label{fig:nakayama}%
	\end{figure}
	
The corresponding Brauer configuration algebras $KQ_{D_n}/{I_{D_n}}$ are symmetric Nakayama algebras.

\subsection{Koszul Brauer configuration algebras}\label{section:koszul}

Let  $D_n=(\rsphere,f_n)$ be the family of dessins given by the \belyi maps $f_n\colon z\mapsto (z^n+1)^2/4z^n$ for $n\geq3$. The zeros of $f_n$ and $f_n-1$ are at the $n$ roots of $-1$ and $1$, respectively, each appearing with multiplicity $2$, and the two poles are at $0$ and $\infty$, both of degree $n$. These dessins are regular polygons of degree $n$ on the sphere.

\begin{figure}[ht]
	\centering
	\includegraphics[scale=0.7]{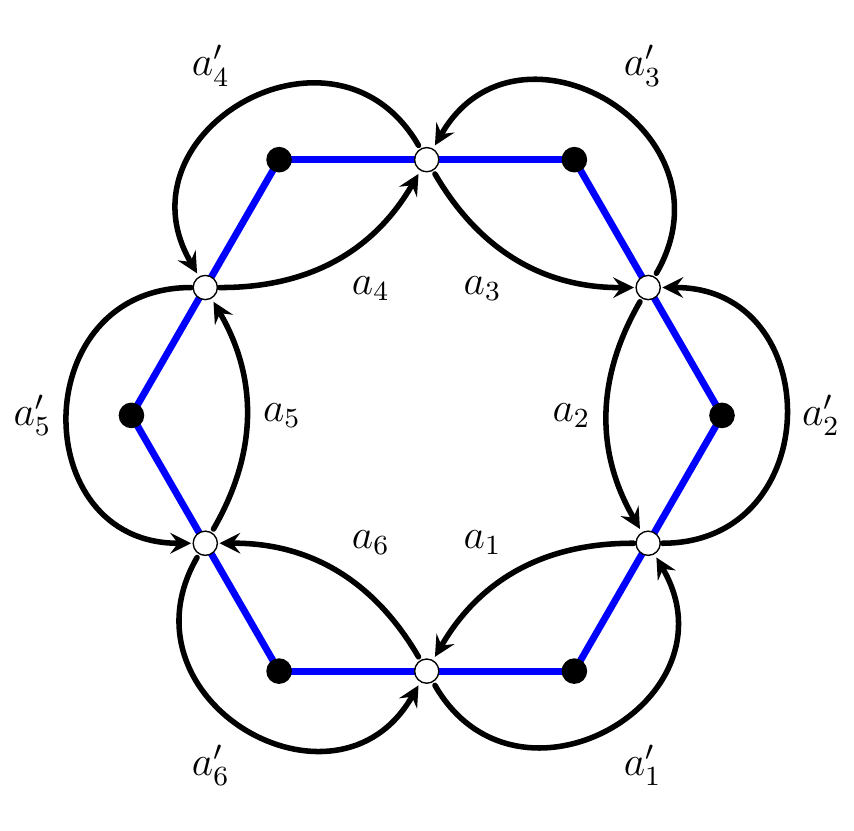}
	\caption{The dessin $D_6$ and its associated quiver. The special cycles are given by the paths $a_ia'_i$ and $a'_ia_i$ for $i=1, \dots, 6$.}%
	\label{fig:koszul}%
	\end{figure}
	
Note in Figure \ref{fig:koszul} that the elements generating $I_{D_6}$ are all of length 2. Evidently this observation extends to $I_{D_n}$. Ideals generated by elements of length 2 are called \emph{quadratic} and the corresponding Brauer configuration algebra is a Brauer graph algebra and by \cite{GSS} it is Koszul.

\section{Brauer configuration algebra of the dual dessin}

\subsection{Paths formed by the faces of a dessin} In this subsection we analyse the paths formed by the faces of a dessin.

\begin{proposition}\label{proposition:faces}The faces of a dessin $D=\per$ decompose $Q_D$ into disjoint cycles oriented clockwise. Moreover, if we count the formal loop arrows, then the lengths of these cycles correspond to the degrees of the faces.\end{proposition}

\begin{proof}Consider a face of $D$ corresponding to the cycle $\varphi_j=(i_1~\cdots ~i_k)$ of $\varphi$. Let $\alpha_{h}$ denote the white vertex of $D$ with $h$ as its half-edge. Each half-edge $i_m$ for $m=1, \dots, k$ induces a single arrow (possibly a formal loop arrow) whose source and target are the white vertices $\alpha_{i_m}$ and $\alpha_{i_m^\sigma}$, respectively. Note that if $\alpha_{i_m}\neq\alpha_{i_m^\sigma}$, then $\alpha_{i_m^\sigma}$ follows $\alpha_{i_m}$ in the counter-clockwise order around $\varphi_j$ because
$i_m^\sigma=i_m^{\inv\varphi\inv\alpha}$ and $\alpha_{i_m^{\inv\varphi\inv\alpha}}=\alpha_{i_m^{\inv\varphi}}$. Therefore, $\varphi_j$ induces a cycle of arrows oriented clockwise. As $i_m$ induces exactly one arrow, and each $i_m$ belongs to exactly one cycle of $\varphi$, the cycles of arrows obtained from the faces of $D$ are disjoint with lengths equal to the degrees of the faces.\end{proof}

When a dessin $D$ has a polygonal face, that is a face of degree $d\geq 2$ such that each black and each white vertex contributes exactly 2 half-edges, then the corresponding paths are formed by concatenations of relations of type 3.

\begin{proposition}\label{proposition:polygonalpath}The polygonal faces of $D$ of degree $d\geq 2$ give rise to cyclic permutations of paths of the type $a_1a_2a_3\cdots a_{d-1}a_d$ such that there are no repeating arrows, and $a_ia_{i+1}$ and $a_da_1$ are relations of type 3, for all $i=1,\dots,{d-1}$.
\end{proposition}

\begin{proof}A polygonal face of degree $d\geq 2$ is incident to $d$ black vertices $\sigma_1$, $\dots$, $\sigma_d$, and $d$ white vertices $\alpha_1$, $\dots$, $\alpha_d$ ordered counter-clockwise, so that $\sigma_1$ follows $\alpha_1$ in this order. Let $a_i$ be the arrow with $\alpha_i$ as source and $\alpha_{i+1}$ as target belonging to the special $\sigma_{i}$-cycle at $\alpha_i$, for $i=1,\dots,d-1$, and let $a_d$ be the arrow with $\alpha_d$ as source and $\alpha_1$ as targed belonging to the special $\sigma_d$-cycle at $\alpha_d$, see Figure \ref{fig:facecorrespondence}. Then $a_ia_{i+1}$ and $a_da_1$ are not subpaths of any special cycle and hence a relation of type 3, for all $i=1,\dots,d-1$.
\end{proof}

\begin{figure}[ht]
\centering
\includegraphics[scale=.9]{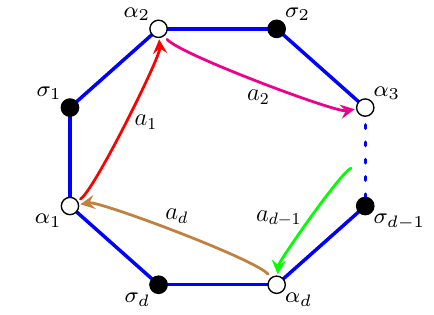}
\caption{In a polygonal face of degree $d\geq 2$ pairs of consecutive arrows belong to distinct $\sigma$-cycles.}%
\label{fig:facecorrespondence}%
\end{figure}

Note that not every cyclic path in which consecutive arrows are type 3 relations corresponds to a face, see Figure \ref{fig:nontype3} for an example.

\begin{figure}[ht]
\centering
\includegraphics[scale=.9]{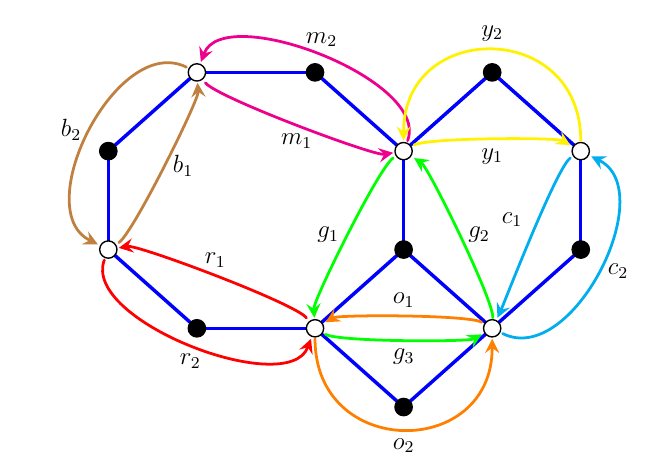}
\caption{The path $r_1b_1m_1y_1c_1o_1$ has no repeated arrows, and all of its subpaths of length 2 as well as $o_1r_1$ are relations of type 3, yet it corresponds to no face of the underlying dessin.}%
\label{fig:nontype3}%
\end{figure}

In general, a face of $D$ with vertices $\sigma_1, \dots, \sigma_k$ of degree at least 2 gives rise to paths of type $A_1A_2\cdots A_k$, where $A_j$ ($j=1,\dots, k$) is a path of arrows in the corresponding $\sigma_j$-cycle, and $A_j$ and $A_{j+1}$ ($j=1,\dots,k-1$) are connected by a relation of type 3, as well as $A_k$ and $A_1$. 

\subsection{Dual dessins} Consider a dessin $D=(X,f)$. Its \emph{dual} dessin $D^*$ is defined as the dessin corresponding to the \belyi pair $(X,1/f)$.

In terms of permutation representations, if $D=\per$, then $D^*$ will have the triple $(\inv\varphi,\inv\alpha,\inv\sigma)$ as its permutation representation. Geometrically this means that the black vertices and the face centres of the dual are the face centres and the black vertices of $D$, respectively, while the white vertices remain unchanged, except for the orientation of the labels. The half-edges of $D^*$ are the curved segments that connect the face centres and the white vertices of $D$, see Figure \ref{dual} for an example.

\begin{figure}[ht]
  \centering
  \includegraphics[scale=.5]{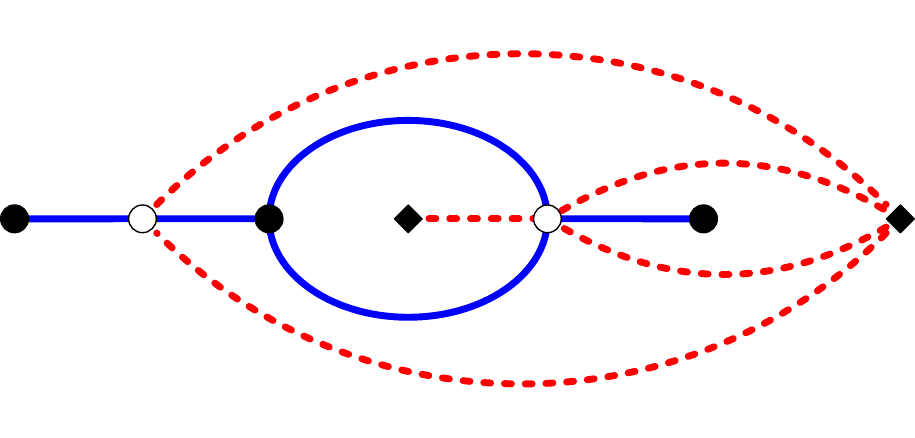}
  \caption{A dessin (full) and its dual (dashed). The black vertices of the dual are indicated by $\blacklozenge$.}\label{dual}
\end{figure}

\begin{theorem} Let $D$ be a dessin and let $D^*$ be its dual dessin. Then the quivers $\mathring Q_D$ and $\mathring Q_{D^*}^{op}$ are equal. Furthermore, if $D$ has no vertices and no faces of degree 1 then the quivers $Q_D$ and $Q_{D^*}^{op}$ are equal.
\end{theorem}

\begin{proof} Consider a face of the dessin $D$. By proposition \ref{proposition:faces} the cycle of arrows in $\mathring Q_D$ corresponding to this face is oriented counter-clockwise around the black vertex of $D^*$ dual to this face, see Figure \ref{fig:dual-proof}. Therefore, in order to obtain the quiver of $\mathring Q_{D^*}$ we have to reverse the direction of every arrow in $Q_D$.

If $D$ has no vertices and no faces of degree 1, then formal loops do not appear in $Q_D$ and $Q_{D^*}$ and therefore $Q_D=\mathring Q_D$ and $Q_{D^*}=\mathring Q_{D^*}$. 
\begin{figure}[ht]
\centering
\includegraphics[scale=.6]{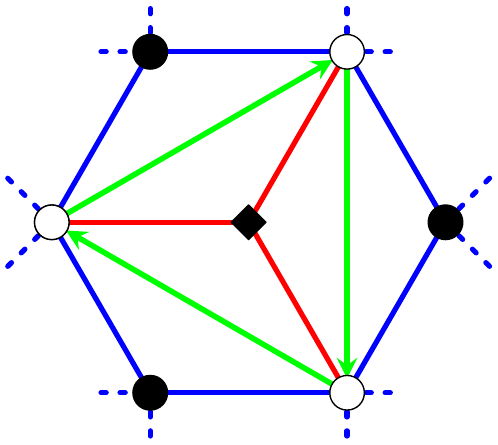}
\caption{ The arrows in a face of a dessin (in green) form a cycle oriented clockwise around the corresponding dual vertex. The dual half-edges are shown in red.}%
\label{fig:dual-proof}%
\end{figure}
\end{proof}

\begin{example} 

Let $D$ be the dessin given by the bipartite 6-gon, as in Figure \ref{fig:example6gon}. The set of arrows of its quiver is decomposed into two disjoint cycles $a_1b_1c_1$ and $a_2c_2b_2$. Its dual dessin $D^*$ is shown in the same Figure. The quiver of $D^*$ is not shown in the Figure, but the reader can verify that its quiver is obtained by reversing the orientation of the arrows in $Q_D$, and the set of arrows of $Q_{D^*}$ is decomposed into three disjoint cycles $a^{\textnormal{op}}_1a^{\textnormal{op}}_2$, $b^{\textnormal{op}}_1b^{\textnormal{op}}_2$ and $c^{\textnormal{op}}_1c^{\textnormal{op}}_2$.

\begin{figure}[ht]
\centering
\includegraphics[trim={1.75cm 1.75cm 0 1.75cm},scale=.6]{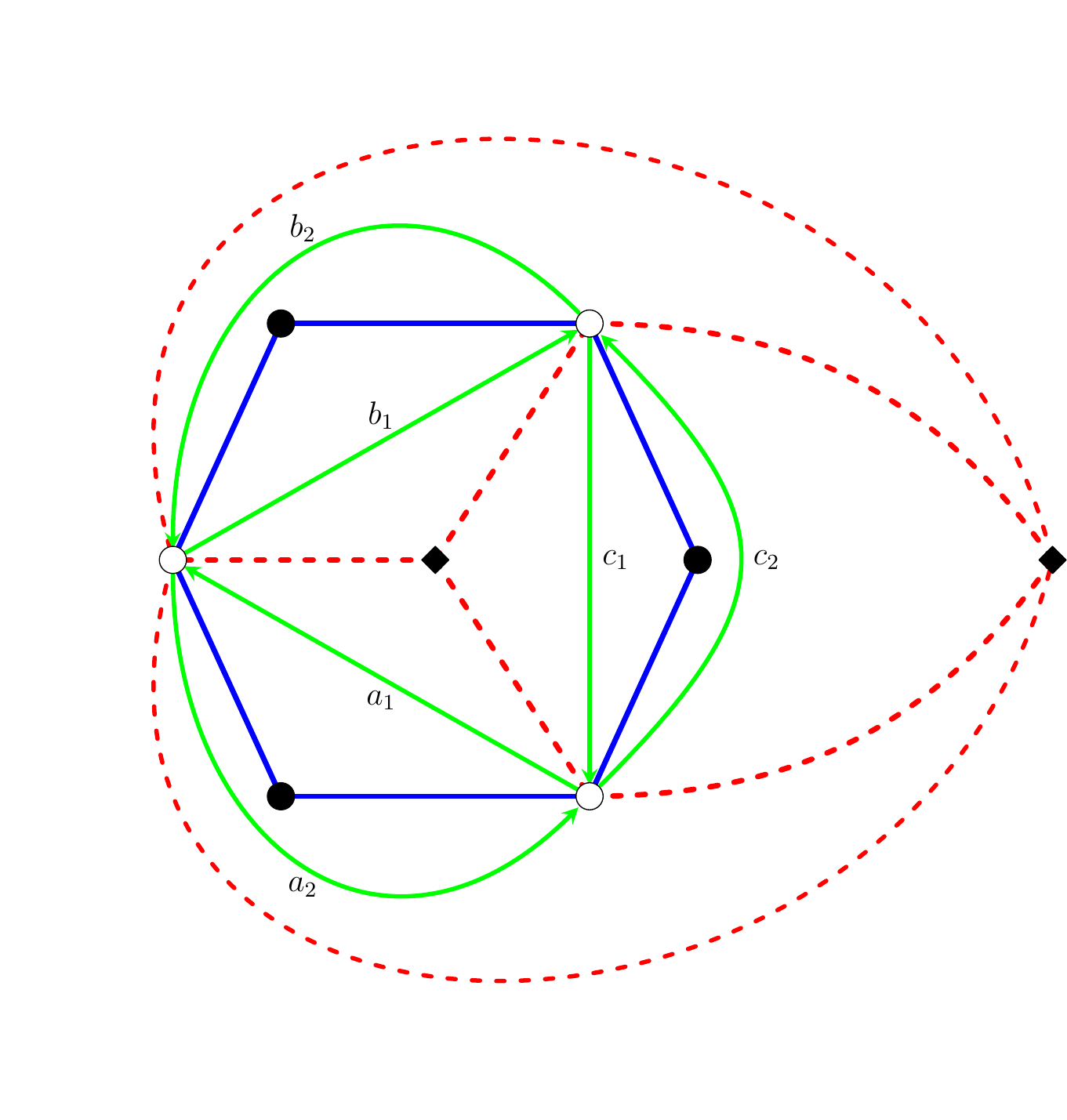}
\caption{The 6-gon dessin (blue), its quiver (green), and its dual dessin (dashed red).}%
\label{fig:example6gon}%
\end{figure}

\end{example}

\begin{example}The dessins $D_n=(\rsphere,z^n)$ are self-dual, i.e.\ $D_n^* \cong D_n$. The Brauer configuration algebra $KQ_{D_n^*}/I_{D^*_n}$ associated to the dual $D^*$ of $D$ is again a symmetric Nakayama algebra. In fact, we have that $KQ_{D_n} / I_{D_n} \cong  KQ_{D_n^*}/I_{D^*_n}$.\end{example}

\section{Galois action on Brauer configuration algebras}

\begin{definition}Let $\mathfrak S$ denote the set of all $\sigma$-cycles. We say that a $\sigma_i$-cycle and a $\sigma_j$-cycle are equivalent if one is a cyclic permutation of the other. An equivalence class represented by a $\sigma_j$-cycle will be denoted $\overline{\sigma_j}$ and let $\overline{\mathfrak S}$ be the set of equivalence classes of $\sigma_j$-cycles.\end{definition}

\begin{lemma}\label{lemma:invariants}Let $D$ be a dessin and $Q_D=(Q_0,Q_1)$ its quiver. The following are Galois invariants:
\begin{itemize}
	\item[(i)] The number of vertices $|Q_0|$ of $Q$.
	\item[(ii)] The number of arrows $|Q_1|$ of $Q$.
	\item[(iii)] The cardinal $|\overline{\mathfrak S}|$ of the set $\overline{\mathfrak S}=\{\overline{\sigma_1},\dots,\overline{\sigma_r}\}$ of equivalence classes of $\sigma$-cycles.
	\item[(iv)] The set $\{|\overline{\sigma_1}|,\dots,|\overline{\sigma_r}|\}$ recording the number of arrows in representatives of equivalence classes of $\sigma$-cycles.
\end{itemize}
\end{lemma}

\begin{proof}That (i-iii) are Galois invariants follows from $|Q_0|$, $|Q_1|$ and $|\mathfrak S|$ being equal to the number of white vertices, the total number of half-edges, and the number of black vertices of $D$, respectively, all of which are Galois invariants. Likewise, (iv) is a Galois invariant because $|\sigma_j|$ is equal to the degree of the black vertex $\sigma_j$, and the degree sequence of $D$ is a Galois invariant.
\end{proof}

Let $D=\per$ be a dessin, $\Lambda_D=KQ_D/I_D$ its associated Brauer configuration algebra. A basis of $\Lambda_D$ is given by the set
\[\{p+I_D\mid p\textnormal{ is a subpath of some $\sigma_j$-cycle in }Q_D\}.\]
The dimension of $\Lambda_D$ is given by the sum
\[\dim_K\Lambda_D=2|Q_0|+\sum_{\substack{\overline{C_i}\in\overline{\mathfrak S}}}|\overline{C_i}|(|\overline{C_i}|-1).\]
For a detailed discussion about a basis of a Brauer configuration algebra in general see \cite{GreenSchrollBrauerConfig}.

\begin{example}Let $D$ be the dessin from example \ref{example3}, let $\Lambda_D=KQ_D/I_D$ be the associated Brauer configuration algebra, and for a path $p$ in $Q_D$ let $\overline{p}=p+I_D$. A basis of $\Lambda_D$ is given by the trivial paths $\overline{e_1}$, $\overline{e_2}$, $\overline{e_3}$, $\overline{e_4}$ and $\overline{e_5}$, one for each vertex of $Q_D$, and the non-trivial paths
\begin{itemize}
	\item[(i)] $\overline{r_1}$, $\overline{r_1r_2}$, $\overline{r_1r_2r_3}$, $\overline{r_1r_2r_3r_4}$;\\
	$\overline{r_2}$, $\overline{r_2r_3}$, $\overline{r_2r_3r_4}$, $\overline{r_2r_3r_4r_5}$;\\
	$\overline{r_3}$, $\overline{r_3r_4}$, $\overline{r_3r_4r_5}$, $\overline{r_3r_4r_5r_1}$;\\
	$\overline{r_4}$, $\overline{r_4r_5}$, $\overline{r_4r_5r_1}$, $\overline{r_4r_5r_1r_2}$;\\
	$\overline{r_5}$, $\overline{r_5r_1}$, $\overline{r_5r_1r_2}$, $\overline{r_5r_1r_2r_3}$;\\	
	$\overline{g_1}$, $\overline{g_2}$;\\
	$\overline{b_1}$, $\overline{b_2}$;
	\item[(ii)] $\overline{r_1r_2r_3r_4r_5}$, $\overline{r_2r_3r_4r_5r_1}$, $\overline{r_3r_4r_5r_1r_2}$, $\overline{r_4r_5r_1r_2r_3}$, $\overline{g_2g_1}$.
\end{itemize}
The paths (i) are images of proper subpaths of $\sigma_j$-cycles, however note that in (ii) we have $\overline{r_3r_4r_5r_1r_2}=\overline{g_1g_2}=\overline{b_1b_2}$, $\overline{r_4r_5r_1r_2r_3}=\overline{r_5r_1r_2r_3r_4}$ and $\overline{g_2g_1}=\overline{b_2b_1}$. The dimension of $\Lambda_D$ is $2\cdot 5+5\cdot 4+2+2=34$.
\end{example}

\begin{theorem}Let $D=\per$ be a dessin and $\Lambda_D=KQ_D/I_D$ the associated Brauer configuration algebra. The dimension $\dim_K\Lambda_D$ of $\Lambda_D$ is a Galois invariant.
\end{theorem}

\begin{proof}The dimension $\dim_K\Lambda_D$ is given by the sum 
\[\dim_K\Lambda_D=2|Q_0|+\sum_{\substack{\overline{C_i}\in\overline{\mathfrak S}}}|\overline{C_i}|(|\overline{C_i}|-1).\]
By the lemma above, the numbers $|Q_0|$ and $|\overline{C_i}|$ are Galois invariant. Therefore, the dimension of $\Lambda_D=KQ_D/I_D$ is a Galois invariant. 
\end{proof}

A basis of the centre $Z(\Lambda_D)$ of $\Lambda_D=KQ_D/I_D$ is given by the set
\[\{1_\Lambda + I_D\}\cup\{p+I_D\mid p\text{ is a special cycle in $Q_D$}\}\cup \mathscr L_D,\]
where $1_\Lambda=e_1+\cdots+e_{|Q_0|}$ is the identity of $\Lambda_D$ and $\mathscr L_D$ is the set of loops in $Q_D$. The dimension $\dim_KZ(\Lambda_D)$ is given by the sum
\[\dim_KZ(\Lambda_D)=1+|Q_0|+|\mathscr L_D|.\]
For a detailed discussion about the centre of a Brauer configuration algebra see \cite{Sierra}.

\begin{example}Let $D$ once again be the dessin from example \ref{example3}, let $\Lambda_D=KQ_D/I_D$ be the associated Brauer configuration algebra, and, for a path $p$ in $Q_D$, let $\overline{p}=p+I_D$. A basis of $Z(\Lambda_D)$ is given by the identity $\overline{1}_\Lambda$, the loop $\overline{r_4}$ and the special cycles
\[\overline{r_1r_2r_3r_4r_5},
	\overline{r_2r_3r_4r_5r_1},
	\overline{r_3r_4r_5r_1r_2},
	\overline{r_4r_5r_1r_2r_3},
	\overline{g_2g_1}.\]
	Note that $\overline{r_3r_4r_5r_1r_2}=\overline{b_1b_2}=\overline{g_1g_2}$, $\overline{r_4r_5r_1r_2r_3}=\overline{r_5r_1r_2r_3r_4}$ and $\overline{g_2g_1}=\overline{b_2b_1}$. The dimension is therefore $1+5+1=7$.
\end{example}

\begin{proposition}\label{proposition:basisdimension}Let $D$ be a dessin and $\Lambda_D=KQ_D/I_D$ the associated Brauer configuration algebra. The dimension of the centre $\dim_KZ(\Lambda_D)$ of $\Lambda_D$ is a Galois invariant.
\end{proposition}

\begin{proof}The dimension $\dim_KZ(\Lambda_D)$ is given by the sum
\[\dim_KZ(\Lambda_D)=1+|Q_0|+|\mathscr L_D|.\]
By lemma \ref{lemma:invariants}, all the summands are invariants and therefore the dimension $\dim_KZ(\Lambda_D)$ of the centre $Z(\Lambda_D)$ of $\Lambda_D$ is a Galois invariant. 
\end{proof}

In fact, a stronger result holds, namely Galois-conjugate dessins have isomorphic centres. In order to prove this, we need the following technical result. 

\begin{lemma}\label{lemma:basis}Let $\Lambda=kQ/I$ be a Brauer configuration algebra with multiplicity function $\mu=1$. Let $B$ be the basis of $Z(\Lambda)$ given by the identity, the loops and the special cycles of $Q$.
If $\bar p,\bar q\in B$ such that neither is the identity, then $\bar p \bar q =\bar 0$.
\end{lemma}

\begin{proof}If $p$ and $q$ are not loops and $t(p) = s(q)$ then $p-q$ is necessarily a relation in $I$. Therefore $\bar p =\bar q$ and $\bar p \bar q = \bar p \bar p = 0$ as $pp$ necessarily contains a type 2 relation.

If $p$ is a loop and $t(p)=s(q)$ or $t(q)=s(p)$, then $q$ is a basis element with $p$ as either the first or the final arrow. The product $\bar p \bar q$ is $\bar 0$ in the former because $p^2$ is a relation, and in the latter because $pq$ is a type 2 relation.  
\end{proof}

\begin{theorem}Let $D_1$ and $D_2$ be Galois-conjugate dessins with Brauer configuration algebras $\Lambda_1=kQ_1/I_1$ and $\Lambda_2=kQ_2/I_2$, respectively. Then $Z(\Lambda_1)\cong Z(\Lambda_2)$.
\end{theorem}

\begin{proof}Let $B_1$ and $B_2$ be the bases of $Z(\Lambda_1)$ and $Z(\Lambda_2)$, respectively, given by the identity, the loops and the special cycles. By proposition \ref{proposition:basisdimension} we have $\dim Z(\Lambda_1) = \dim Z(\Lambda_2)=n$ so there is a vector-space isomorphism $f\colon Z(\Lambda_1)\to Z(\Lambda_2)$ mapping $B_1$ to $B_2$ so that loops are mapped to loops and special cycles to special cycles.

Let $\bar p_i\in B_1$ for $i=2,\dots,n$ be the basis elements of $Z(\Lambda_1)$ which are not $\bar 1_{\Lambda_1}$. Then for $v,w\in Z(\Lambda_1)$ we have
\begin{align*}
v &=\alpha \bar 1_{\Lambda_1} + \sum_{i=2}^n \alpha_i \bar p_i,\\
w &=\beta \bar 1_{\Lambda_1} + \sum_{i=2}^n \beta_i \bar p_i,
\end{align*}
where $\alpha$, $\beta$, $\alpha_i$ and $\beta_i$ for $i=2,\dots,n$ are scalars. Then
\[vw=\alpha\beta \bar 1_{\Lambda_1} + \sum_{i=2}^n(\alpha_i\beta+\alpha\beta_i)\bar p_i\]
due to lemma \ref{lemma:basis}. Then
\begin{align*}
f(vw)&=\alpha\beta \bar 1_{\Lambda_2} + \sum_{i=2}^n(\alpha_i\beta+\alpha\beta_i)f(\bar p_i)\\
&=\left(\alpha f(\bar 1_{\Lambda_1}) + \sum_{i=2}^n \alpha_i f(\bar p_i)\right)\left(\beta f(\bar 1_{\Lambda_1}) + \sum_{i=2}^n \beta_i f(\bar p_i)\right)=f(v)f(w)\end{align*}
again due to lemma \ref{lemma:basis}.
\end{proof}

\end{document}